\documentclass[english, 10pt]{article}
\usepackage{amsmath}
\usepackage{graphicx}
\usepackage{amssymb}
\usepackage{amsfonts, amsthm}
\usepackage[colorlinks=true,linkcolor=blue,citecolor=blue,urlcolor=blue]{hyperref}

\newcommand{\al}{\alpha}

\newcommand{\de}{\delta}

\newcommand{\auskommentieren}[1]{}

\newcommand{\beq}{\begin{equation}}
\newcommand{\eeq}{\end{equation}}
\newcommand{\bea}{\begin{equation}\begin{aligned}}
\newcommand{\eea}{\end{aligned}\end{equation}}

\newtheorem{theorem}{Theorem}[subsection]

\theoremstyle{definition}
\newtheorem{defi}[theorem]{Definition}

\usepackage{array} 
\usepackage{amssymb}
\usepackage{latexsym}
\usepackage{epic}
\usepackage{curves}
\usepackage{fancyhdr}
\usepackage{xspace}

\numberwithin{equation}{section}
\numberwithin{theorem}{section}

\title{Non-collapsing for nearly spherical closed convex curves under power curvature flow}
\author{Heiko Kr\"oner}
\begin{document}
\maketitle
\footnotetext{
\textsc{Eberhard Karls Universit\"at, Mathematisches Institut, 
Auf der Morgenstelle 10,
D-72076 T\"ubingen,
Germany}\\
\textit {E-mail:} \href{mailto:kroener@na.uni-tuebingen.de}{kroener@na.uni-tuebingen.de} \\
\textit {Url:} \href{http://na.uni-tuebingen.de/~kroener/}{http://na.uni-tuebingen.de/$\sim$kroener/}}

\begin{abstract}
We show non-collapsing for the evolution of nearly spherical closed convex curves in $\mathbb{R}^2$ under power curvature flow
\beq \label{0}
\dot x = -\kappa^p\nu, \quad p>1,
\eeq
where $\nu$ is the unit normal and $\kappa$ the curvature,
using two-point-methods as in \cite{ALM}. 
\end{abstract}

We recall the definition of $\de$-non-collapsing, cf. e.g. \cite{A}.
\begin{defi}
A mean-convex hypersurface $M$ bounding an open region $\Omega$ in $\mathbb{R}^{n+1}$ is $\de$-non-collapsed (on the scale of the mean curvature) if for every $x\in M$ there is an open ball $B$ of radius $\frac{\de}{H(x)}$ contained in $\Omega$ with $x \in \partial \Omega$.
\end{defi}
It was proved in \cite{SW} that any compact mean-convex solution of the mean curvature flow is $\de$-non-collapsed for some $\de>0$ and closely related statements are deduced in \cite{W}; in both cases a lengthy analysis of the properties of the mean curvature flow is needed for the proof. 

In a very short paper, which uses only two-point-methods and the maximum principle,  Andrews shows in $\cite{A}$ that a mean-convex, closed, embedded and $\de$-non-collapsed initial hypersurface remains  $\de$-non-collapsed under mean curvature flow. In \cite{ALM} this result is extended to fully nonlinear homogeneous degree one, concave or convex, normal speeds.

In this short note we consider the so-called power curvature flow (PCF) for closed convex curves in $\mathbb{R}^2$ given by the equation (\ref{0}). This is a special case of the flow
\beq \label{10}
\dot x = -H^p\nu, \quad p>1,
\eeq
for closed mean-convex hypersurfaces in $\mathbb{R}^{n+1}$. 
It was proved in \cite{S} that the flow (\ref{10}) (for all $n$) shrinks a convex initial hypersurface to a point and in \cite{SS} that after a proper rescaling a nearly umbilical (i.e. the ratio of the biggest and the smallest principal curvature is sufficiently close to 1) initial hypersurface converges in $C^{\infty}$ to a unit sphere.

Our aim is to prove Theorem \ref{5} using the two-point-method from \cite{ALM}. 

\begin{theorem} \label{5}
Let  $X: \mathbb{S}^1\times [0,T) \rightarrow \mathbb{R}^{n+1}$ be a family of smooth convex embeddings evolving by PCF (\ref{0}). There is $\mu_0=\mu_0(p)>1$ so that if $M_0= X( M, 0)$ is $\mu$-non-collapsed with some $1\le \mu \le \mu_0$, then
$M_t= X( M, t)$ is $\mu$-non-collapsed for all $t \in [0, T)$.
\end{theorem}
\begin{proof}
W.l.o.g. we can assume $1<\mu\le \mu_0$, otherwise consider a sequence $1<\mu_k\rightarrow 1$.
From \cite[Lemma 2.3.4]{G} we get the evolution equation
\beq
\frac{d}{dt}(\kappa^p) -p\kappa^{p-1}\Delta \kappa^p = p \kappa^{p-1} \kappa^{2+p}
\eeq
and after a short calculation
\beq
\frac{d}{dt}\kappa - p \kappa^{p-1}\Delta \kappa = \kappa^{2+p} + p(p-1)\kappa^{p-2}\|D\kappa\|^2.
\eeq
Following \cite{ALM} we define
\beq
Z(x,y,t) := 2\frac{\left<X(x,t)-X(y,t), \nu(x,t)\right>}{\|X(x,t)-X(y,t)\|^2} = 2 \left<w, \nu_x\right>d^{-1}, 
\eeq
for $t \in [0,T)$, $(x,y) \in (\mathbb{S}^1 \times \mathbb{S}^1) \backslash D$, where $D = \{(x,x): x \in \mathbb{S}^1\}$; we use the abbreviations
$d=\|X(x,t)-X(y,t)\|$, $w=d^{-1}\left(X(x,t)-X(y,t)\right)$, $\partial^x=\frac{\partial X}{\partial x}(x, t)$, $\nu_x = \nu(x,t)$, $\kappa_x=\kappa(x,t)$, etc; the sub- or superscript $x$ (in contrast to $y$) will be omitted sometimes.
The supremum of $Z$ with respect to $y$ gives the curvature of the largest interior sphere which touches at $x$.

We show that 
\beq \label{6}
w(x,y,t) = Z(x,y,t) - \mu \kappa(x,t) \le 0
\eeq
for $0\le t <T$.

In view of the assumptions (\ref{6}) holds for $t=0$.
We argue by contradiction. Let $\de>0$ be small, assume $\sup w(\cdot , \cdot , t)=\de$ for a $0<t  <T$ and choose $t$ minimal with these properties. Let $x,y \in \mathbb{S}^1$, so that $w(x, y, t)=\de$, then $x\neq y$. We choose normal coordinates $(x)$ and $(y)$ around $x$ and $y$, respectively, and obtain in $(x, y, t)$ by adapting the equations \cite[(8)-(11)]{ALM} 
\bea \label{1}
0 \le & \dot w - p \kappa^{p-1}\left(\frac{\partial ^2}{\partial x^2}w+2\frac{\partial ^2}{\partial x\partial y}w+\frac{\partial ^2}{\partial y^2}w\right) \\
=&  -\mu \frac{\partial}{\partial t}\kappa-\frac{2}{d^2}\kappa^p + \frac{2}{d^2}\kappa_y^p\left<\nu_y, \nu-dZw\right> + \frac{2}{d}\left<w, \nabla (\kappa^p)\right>\\&
+Z^2\kappa^p +\frac{2p}{d^2} \kappa^{p-1}\left(Z-\kappa\right) + p \kappa^{p+1}Z\\
&- \frac{2p}{d}\kappa^{p-1}\nabla \kappa\left<w, \partial^x\right> - p \kappa^{p}Z^2\\
&+ \frac{4p \mu}{d}\kappa^{p-1}\nabla\kappa\left<w, \partial^x\right>
+ \mu p \kappa^{p-1}\frac{\partial^2\kappa}{\partial x^2} \\
&-\frac{4p}{d^2}\kappa^{p-1}\left(Z-\kappa \right)\left<\partial^y, \partial^x\right>
-\frac{4p\mu }{d}\kappa^{p-1}\nabla \kappa\left<w, \partial^y\right> \\
&+\frac{2p}{d^2}\kappa^{p-1}\left(Z-\kappa_y\right)\\
= & -\mu \kappa^{p+2} - \mu p(p-1)\kappa^{p-2}\|D\kappa\|^2 + p \kappa^{p+1}Z\\
& - \frac{2(1+p)}{d^2}\kappa^p + \frac{4p}{d^2}\kappa^{p}\left<\partial^y, \partial^x\right> \\
&+ \frac{2}{d^2}\kappa_y^p-\frac{2p}{d^2}\kappa^{p-1}\kappa_y \\
& + \frac{4p}{d^2}\kappa^{p-1}Z-\frac{4p}{d^2}\kappa^{p-1}Z\left<\partial^y, \partial ^x\right>\\
&+ (1-p)\kappa^pZ^2 \\
&+ \frac{4p\mu}{d}\kappa^{p-1}\nabla \kappa\left< w, \partial^x-\partial ^y\right>.
\eea
The second line of the right-hand side of equation (\ref{1}) can be rewritten as
\beq
-\frac{4p}{d^2}\kappa^p\left(\frac{1+p}{2p}-\left<\partial^y, \partial^x\right>\right)
\eeq
and the fourth line as
\beq
\frac{4p}{d^2}Z\kappa^{p-1}\left(1-\left<\partial^y, \partial^x\right>\right).
\eeq
From 
\beq
\frac{\partial w}{\partial x}(x, y, t)=0
\eeq
we conclude
\beq \label{2}
\nabla \kappa = \frac{2}{\mu d}\left(\kappa -Z\right)\left<w, \partial^x\right>.
\eeq
so that the right-hand side of (\ref{1}) can be written as
\bea \label{3}
& -(\mu\kappa-pZ)\kappa^{p+1} - \mu p(p-1)\kappa^{p-2}\|D\kappa\|^2 \\
&+ \frac{2}{d^2}\kappa_y^p-\frac{2p}{d^2}\kappa^{p-1}\kappa_y \\
&+ (1-p)\kappa^pZ^2 \\
& +\frac{4p}{d^2}\kappa^{p-1} \left(Z-\kappa\right) \\
& \quad \quad \left(1-\left<\partial^y, \partial^x\right>+2\left<w, \partial^y-\partial^x\right>\left<w, \partial^x\right>+\frac{1-p}{2p} \right) \\
&+ \frac{2(p-1)}{d^2}Z\kappa^{p-1}. \\
\eea

Using 
\beq
\de \ge \sup Z(y, \cdot)-\mu \kappa_{y} = Z(x, y)-\mu \kappa_{y}
\eeq
we get
\beq
\kappa_x\le \kappa_y\le Z.
\eeq
Let us write $\kappa_y=(1+\eta+\frac{\tilde \de}{\kappa_x})\kappa_x$ with suitable $ 0 \le \eta\le \mu-1$ and $0 \le \tilde \de \le \de$.  We estimate the second line of (\ref{3}) from above by
\bea \label{11}
\frac{2}{d^2}\kappa^p& (1+\eta+\frac{\tilde \de}{\kappa_x})\left((1+\eta+\frac{\tilde \de}{\kappa_x})^{p-1}-p\right) \\
\le& \frac{2}{d^2}\kappa^p\left(1-p+(p-1)(1+\frac{p-2}{2})(\eta+\frac{\tilde \de}{\kappa_x}) ^2 + c_1(p) (\eta+\frac{\tilde \de}{\kappa_x})^3 \right),
\eea
$c_1(p)>0$ a constant, by Taylor's expansion.
We have
\beq \label{12}
\frac{4p}{d^2}\kappa^{p-1}(Z-\kappa)\frac{1-p}{2p} = \frac{2}{d^2}(1-p)(\mu-1+\frac{\de}{\kappa})\kappa^p.
\eeq
The  proof of \cite[Lemma 6]{ALM} shows that there are $\al \in [0, \frac{\pi}{2}]$ and normal coordinates $(x)$ and $(y)$ so that
\beq
|\left<w, \nu_x\right>| = \sin \al \quad \wedge\quad \left<\partial^y, \partial^x\right>=-\cos 2\al
\eeq
 and
\beq \label{7}
1-\left<\partial^y, \partial^x\right>+2\left<w, \partial^y-\partial^x\right>\left<w, \partial^x\right> = -2\cos^2 \al.
\eeq
There holds
\beq
d<\frac{1}{Z} \Rightarrow \al \le \frac{\pi}{4}.
\eeq
Estimating (\ref{3}) using (\ref{11}), (\ref{12}) and (\ref{7}) leads to
\bea \label{8}
0 \le & (p-1)\kappa^{p+2}\mu (1-\mu+O(\de)) \\
&+ \frac{2}{d^2}\kappa^p \left((p-1)(1+\frac{p-2}{2})\eta^2 + c_1(p) \eta^3+O(\de) \right)\\
& +\frac{4p}{d^2}\kappa^{p-1} \left(Z-\kappa\right) \\
& \quad \quad \left(1-\left<\partial^y, \partial^x\right>+2\left<w, \partial^y-\partial^x\right>\left<w, \partial^x\right> \right). 
\eea
Assume that $\de>0$ is small.
If $d \ge \frac{1}{Z}$ we use that $Z-\kappa\ge 0$ and (\ref{7}) to see that the right-hand side of (\ref{8}) is negative provided $\mu_0>1$ is sufficiently close to 1. 

Let us assume $d<\frac{1}{Z}$. Then the last summand in (\ref{8}) is estimated from above by
\beq
-\frac{4p}{d^2}\kappa^p(\mu-1+O(\de)).
\eeq
\end{proof}

\end{document}